\documentclass{amsart}
\usepackage[english]{babel}
\usepackage{amsmath}
\usepackage{amsfonts}
\usepackage{amssymb}
\usepackage{graphicx}
\usepackage{mathrsfs}
\usepackage{url}

\usepackage{makecell}
\usepackage{tabu}
\usepackage{mathtools}
\DeclarePairedDelimiter\ceil{\lceil}{\rceil}

\makeatletter
\newcommand*{\rom}[1]{\expandafter\@slowromancap\romannumeral #1@}
\makeatother

\usepackage{multirow}

\newcommand{\re}{\mathbb{R}}

\newcommand{\bE}{\mathbb{E}}
\newcommand{\N}{\mathbb{N}}
\def\af{\alpha}

\newcommand{\st}{\mathrm{s.t.}}
\newcommand{\reff}[1]{(\ref{#1})}

\newcommand{\eps}{\epsilon}

\newcommand{\mt}[1]{\mathtt{#1}}
\def\rank{\mbox{rank}}

\newcommand{\bdes}{\begin{description}}
\newcommand{\edes}{\end{description}}

\newcommand{\bal}{\begin{align}}
\newcommand{\eal}{\end{align}}

\newcommand{\bnum}{\begin{enumerate}}
\newcommand{\enum}{\end{enumerate}}

\newcommand{\bit}{\begin{itemize}}
\newcommand{\eit}{\end{itemize}}

\newcommand{\bea}{\begin{eqnarray}}
\newcommand{\eea}{\end{eqnarray}}
\newcommand{\be}{\begin{equation}}
\newcommand{\ee}{\end{equation}}

\newcommand{\baray}{\begin{array}}
\newcommand{\earay}{\end{array}}

\newcommand{\bsry}{\begin{subarray}}
\newcommand{\esry}{\end{subarray}}

\newcommand{\bca}{\begin{cases}}
\newcommand{\eca}{\end{cases}}

\newcommand{\bcen}{\begin{center}}
\newcommand{\ecen}{\end{center}}

\newcommand{\bbm}{\begin{bmatrix}}
\newcommand{\ebm}{\end{bmatrix}}

\newcommand{\btab}{\begin{tabular}}
\newcommand{\etab}{\end{tabular}}

\newtheorem{theorem}{Theorem}[section]

\newtheorem{alg}[theorem]{Algorithm}

\theoremstyle{definition}

\newtheorem{exmp}{Example}[section]

\numberwithin{equation}{section}

\title[Stochastic Polynomial Optimization]
{Stochastic Polynomial Optimization}

\author[Jiawang Nie]{Jiawang~Nie}
\address{Jiawang Nie, Suhan Zhong, Department of Mathematics,
University of California San Diego,
9500 Gilman Drive, La Jolla, CA, USA, 92093.}
\email{njw@math.ucsd.edu,suzhong@ucsd.edu}

\author[Liu Yang]{Liu~Yang}
\address{Liu Yang, school of Mathematics and Computational Sciences,
Xiangtan University, Xiangtan, Hunan, China, 411105.}
\email{yangl410@xtu.edu.cn}

\author[Suhan Zhong]{Suhan~Zhong}

\begin{document}


\date{}


\begin{abstract}
This paper studies stochastic optimization problems with polynomials.
We propose an optimization model with sample averages and perturbations.
The Lasserre type Moment-SOS relaxations
are used to solve the sample average optimization.
Properties of the optimization and its relaxations are studied.
Numerical experiments are presented.
\end{abstract}

\maketitle

\section{Introduction}

Stochastic optimization is about functions that depend on random variables.
A typical stochastic optimization problem is
\be \label{stop:minE(f):K}
\min_{ x \in K} \quad
f(x) \,:=\, \bE [F(x,\xi)]
\ee
where $F: \re^n \times \re^r \to \re$ is a function in
$(x, \xi)$. The variable $\xi$ is a random vector, and
the decision variable $x \in \re^n$ is required
to be contained in a set $K \subseteq \re^n$.
In \reff{stop:minE(f):K}, the symbol $\bE$ denotes
the expectation of a function in the random vector $\xi$.
Frequently used methods for solving stochastic optimization
are often based on sample average approximation (SAA). We refer to
\cite{branda2012sample,fu2015handbook,
kleywegt2002sample,linderoth2006empirical,luedtke2008sample,
plambeck1996sample,robinson1996analysis,santoso2005stochastic,
shalev2009stochastic,shapiro2009lectures,verweij2003sample}
for related work on stochastic optimization.
The SAA methods use sample averages to approximate the expectation function $f(x)$,
transforming the stochastic optimization into deterministic optimization.
Many classical SAA methods assume the objective functions
are convex and are based on evaluations of gradients or subgradients.
They can also be applied to nonconvex problems, however,
the global optimality may not be guaranteed.
There exists relatively less work on nonconvex stochastic optimization
\cite{bastin2006convergence,ghadimi2013stochastic,ghadimi2016mini}.
Generally, nonconvex stochastic optimization problems are
computationally challenging, because the deterministic case is already difficult.

This paper discusses the special case of stochastic polynomial optimization,
i.e., $F(x,\xi)$ is a polynomial function in $x$
and $K$ is a semialgebraic set defined by polynomials.
When $F$ does not depend on $\xi$, this is the case of polynomial optimization.
The Lasserre type Moment-SOS relaxations
are efficient and reliable for solving polynomial optimization \cite{LasBk15,Lau09}.
When $F$ depends on the random vector $\xi$,
the objective $f(x)$ is the expectation of $F(x,\xi)$
with respect to $\xi$. Hence, $f(x)$ is still a polynomial.
When the distribution of $\xi$ is known explicitly (e.g, Gaussian, Poisson, etc),
the objective $f$ can be expressed by integral formula.
However, when the distribution of $\xi$ is not known exactly,
or its density function is too complicated for evaluating the expectation,
it is not practical to get an explicit formula for $f$.
In most methods for stochastic optimization, the objective $f$
is approximated by sample averages.
In this article, we discuss how to use the efficient
Moment-SOS relaxation for solving stochastic optimization
with sample average approximation.

In stochastic polynomial optimization, we assume that
\[
F(x,\xi) :=
\sum_{ \substack{ \af = (\af_1, \ldots, \af_n)  }  }
c_\af(\xi) x_1^{\af_1} \cdots x_n^{\af_n}
\]
is a polynomial in $x \in \re^n$.
Here, each coefficient $c_\af(\xi)$ is a measurable function in $\xi$.
The feasible set $K$ is assumed to be in the form
\be  \label{set:K:g>=0}
K \, := \, \{x \in \re^n:\ g_1(x)\geq 0, \ldots, g_m(x) \geq 0\},
\ee
for given polynomials $g_1, \ldots, g_m$ in $x \in \re^n$.
For each $x$, $F(x,\xi)$ is a measurable function in $\xi$.
So the objective $f(x) :=\bE [F(x,\xi)]$ is also a polynomial.
The stochastic optimization can be expressed as
\be \label{eq:original}
\left\{ \baray{rl}
\min & f(x) :=  \bE\big[ F(x,\xi) \big]   \\
\st & g_1(x)\geq 0, \ldots, g_m(x) \geq 0.
\earay \right.
\ee
The coefficients of the polynomial $f$
are typically not known explicitly,
because the true distribution of $\xi$ is usually not known exactly.
However, they can be estimated by sample average approximation.
In applications, we can generate samples of $\xi$,
say, $N$ random samples $\xi^{(1)},\xi^{(2)},...,\xi^{(N)}$.
The expectation function $f(x) = \bE [F(x,\xi)]$
can be approximated by the sample average
\[
f_N(x) \, := \, \frac{1}{N}\sum\limits_{k=1}^N F(x,\xi^{(k)}).
\]
If each sample $\xi^{(k)}$ obeys the same distribution of $\xi$,
then $\bE [f_N(x)] = f(x)$. Furthermore, when all $\xi^{(k)}$
are independently identically distributed (i.i.d.),
the Law of Large Numbers (LLN) (see \cite{judd1985law}) implies that
\[
f_N(x) \to f(x) \qquad \text{as}\ N \to \infty,
\]
with probability one and under some regularity conditions.
The resulting sample average approximation for \reff{eq:original} is
\be \label{SAA}
\left\{ \baray{rl}
\min & f_N(x) \\
\st & g_1(x)\geq 0, \ldots, g_m(x) \geq 0.
\earay \right.
\ee
This is also a polynomial optimization problem.
It can be solved globally by the Lasserre type
Moment-SOS hierarchy of relaxations \cite{Las01}.

Sample average approximation methods have good statistical properties.
For convenience, denote by $\vartheta^*,\vartheta_N$ the optimal values of
(\ref{eq:original}) and (\ref{SAA}) respectively,
and denote by $S,S_N$ their optimizer sets respectively.
Assume that: i) $f$ is continuous and $S$ is nonempty;
ii) there is a compact set $C\subseteq\mathbb{R}^n$
such that $S \subseteq C$ and $f_N$ converges uniformly to
$f$ on $C$, with probability one;
iii) for all $N$ large enough, $\emptyset\not=S_N\subseteq C$.
Then, it can be shown (see \cite{shapiro2009lectures})
that $\vartheta_N \longrightarrow \vartheta^*$ and
$\mbox{dist}(S_N,S)\longrightarrow 0$
\footnote{For sets $A,B \subseteq  \re^n$, their distance is defined as
\[
\mbox{dist}(A,B) := \max \left\{
\sup_{x\in A}\inf_{y \in B} \|x-y\|,\,
\sup_{x\in B}\inf_{y \in A} \|x-y\|
\right\}.
\] }
as $N\rightarrow \infty$, with probability one. Moreover, if
$\xi^{(1)},\xi^{(2)},...,\xi^{(N)}$ are independently identically distributed
as $\xi$ is, then
$\bE [\vartheta_N] \, \leq \, \bE [\vartheta_{N+1}] \, \leq \, v^*$.
That is, as the sample size $Z$ increases,
the sample average optimization gives better approximation for \reff{eq:original}.
We refer to \cite{shapiro2009lectures} for more details about
the convergence of the sample average optimization.

When $F(x,\xi)$ is a polynomial in $x$, the sample average approximation
\reff{SAA} is also a polynomial optimization problem.
The Lasserre type Moment-SOS relaxations
can be applied to solve it. However,
the following concerns need to be addressed:

\bit

\item For given samples $\xi^{(1)},\ldots, \xi^{(N)}$,
the optimizer set $S_N$ of \reff{SAA}
may (or may not) be far away from the optimizer set $S$ of \reff{eq:original}.
For instance, it is possible that \reff{eq:original} is bounded from below
and has a global minimizer, while \reff{SAA} is unbounded from below
and has no global minimizers.

\item The sample average $f_N(x)$ is only an approximation
for the objective $f(x)$.
Usually, the optimizer sets of \reff{eq:original} and \reff{SAA}
are not exactly same.
Therefore, the sample average approximation \reff{SAA}
does not need to be solved exactly.
However, we still expect that the optimizer set $S_N$ of \reff{SAA}
(if it is nonempty)
is a good approximation for the optimizer set $S$ of \reff{eq:original},
though they might be very different.
Generally, the optimizer set $S$ can not be determined exactly,
unless the objective $f(x)$ can be determined exactly.

\eit

To address the above concerns, we propose the following
perturbation sample average approximation (PSAA) model
\be \label{PSAA:pop}
\left\{ \baray{rl}
\min & f_N(x) + \eps   \| [x]_{2d} \| \\
\st & g_1(x)\geq 0, \ldots, g_m(x) \geq 0,
\earay \right.
\ee
where $d$ is the smallest integer such that
\[
2d \geq \max\{ \deg(f_N), \deg(g_1), \ldots, \deg(g_m) \}
\]
and $[x]_{2d}$ is the vector of monomials in $x$
and with degrees at most $2d$ (see \reff{vec:[x]d}).
The norm $\| \cdot \|$ is the standard Euclidean norm.
We use the Lasserre type Moment-SOS relaxation of degree $2d$
for solving \reff{PSAA:pop}.
A small parameter $\eps >0$ is often selected, for \reff{PSAA:pop}
to approximate \reff{eq:original} well.
In this article, we discuss properties of \reff{PSAA:pop},
as well as its Moment-SOS relaxations.
The perturbation term $\eps \| [x]_{2d} \|$ plays an important role
in sample average approximation.
The paper is organized as follows. We review some basics
for polynomial optimization in Section~\ref{sc:pre}.
The properties of the perturbed sample average optimization \reff{PSAA:pop}
are discussed in Section~\ref{sc:saopt}.
The numerical experiments are given in Section~\ref{sc:num}.

\section{Preliminaries}
\label{sc:pre}

\subsection*{Notation}

The symbols $\mathbb{N},\mathbb{R}$ denote the set of nonnegative integers
and real numbers, respectively. For given $x\in \mathbb{R}^n$ and a real scalar
$r>0$, $B(x,r)$ denotes the closed ball in $\mathbb{R}^n$ centered at $x$ with radius $r$,
under the standard Euclidean norm.
For a real symmetric matrix $X$, we write $X \succeq 0$ (resp., $X \succ 0$)
by meaning that $X$ is positive semidefinite (resp., positive definite).
The symbol $\re[x] := \re[x_1, \ldots, x_n]$ denotes the ring of polynomials
with real coefficients and in $x :=(x_1, \ldots, x_n)$.
For a polynomial $f$, $\deg (f)$ refers to its total degree.
For a tuple of polynomials $p = (p_1, p_2,\ldots, p_m)$,
$\deg (p)$ refers to the maximum of the degrees of $p_i$.
For a degree $d$, $\mathbb{R}[x]_d$ stands for the space of all real polynomials
in $x$ and of degrees no more than $d$.
For a nonnegative integer vector $\alpha=(\alpha_1,...,\alpha_n) \in \N^n$, denote
\[
|\alpha|:=\alpha_1+\cdots+\alpha_n, \quad
x^{\alpha}:=x_1^{\alpha_1}\cdots x_n^{\alpha_n}.
\]
For convenience, denote the monomial power set
\[
\N_d^n:=\{\alpha \in \mathbb{N}^n:\, |\alpha| \leq d
\}.
\]
For a degree $k$, $[x]_k$ denotes the vector of all monomials of
degrees at most $k$, ordered in the graded lexicographic ordering, i.e.,
\be \label{vec:[x]d}
[x]_k:=\bbm 1 & x_1 & \cdots & x_n &
x_1^2 & x_1x_2 & \cdots & x_n^k \ebm^T.
\ee
(The superscript $^T$ denotes the transpose of a vector or matrix.)
For $t\in \mathbb{R}$, $\ceil{t}$ denotes the smallest integer
greater than or equal to $t$.
For a function $p(\xi)$ in a random vector $\xi$,
$\bE [p(\xi)]$ stands for the expectation of $p(\xi)$,
with respect to the distribution of $\xi$.

A polynomial $\sigma$ is said to be a sum of squares (SOS)
if $\sigma = s_1^2+s_2^2+\cdots+s_k^2$,
for some $k \in \N$ and polynomials $s_1,s_2,...,s_k \in \re[x]$.
Clearly, if $\sigma$ is SOS and has degree $2d$,
then each $s_i$ must have degree at most $d$.
We use $\Sigma[x]$ to denote the cone of all SOS polynomials,
and $\Sigma[x]_{2d}$ to denote the truncation of
SOS polynomials in $\re[x]_{2d}$.
Checking whether a polynomial is SOS or not
can be done by solving a semidefintie program \cite{Las01,ParMP}.

For a tuple $g := (g_1, \ldots, g_m)$ of polynomials in $\re[x]$,
its quadratic module is the set
\[
Q(g):=\Sigma[x]+g_1\cdot \Sigma[x]+\cdots+g_m\cdot\Sigma[x].
\]
The $2d$-th truncation of $Q(g)$ is
\[
Q(g)_{2d}:=\Sigma[x]_{2d}+g_1\cdot\Sigma[x]_{2d-\text{deg}(g_1)}
+\cdots+g_m\cdot\Sigma[x]_{2d-\text{deg}(g_m)}.
\]
It clearly holds the nesting relation of containment
\[
 \cdots \subseteq Q(g)_{2d} \subseteq
Q(g)_{2d+2} \subseteq \cdots  Q(g).
\]
Indeed, each $Q(g)_{2d}$ is a convex cone of the space $\re[x]_{2d}$.
The tuple $g$ determines the semialgebraic set in \reff{set:K:g>=0}.
Obviously, if $f \in Q(g)$, then $f \geq 0$ on $K$. To ensure $f \in Q(g)$,
we often require $f > 0$ on $K$.
The quadratic module $Q(g)$ is said to be {\it archimedean}
if there exists a single polynomial $p \in Q(g)$ such that
the inequality $p(x) \geq 0$ defines a compact set in $\re^n$.
If $Q(g)$ is archimedean, then the set $K$ must be compact.
The converse is not necessarily true. However, if
$K$ is compact (say, $K \subseteq B(0,R)$ for some radius $R$),
one can always enforce $Q(g)$ to be archimedean
by adding the redundant polynomial $R^2- \| x \|^2$ to the tuple $g$.
When $Q(g)$ is archimedean, if $f >0$ on $K$,
then we must have $f \in Q(g)$. This conclusion is referred to
Putinar's Positivstellensatz, which was shown in \cite{Put}.
Interestingly, when $f \geq 0$ on $K$,
we still have $f \in Q(g)$, under some optimality conditions \cite{Nie-opcd}.

For a given dimension $n$ and degree $d$, denote by
$\re^{\N_d^n}$ the space of real vectors
that are indexed by $\af \in \N^n_d$, i.e.,
\[
\re^{\N_d^n} \, :=\, \{ y = (y_\af)_{ \af \in \N_d^n } : \, y_\af \in \re \}.
\]
Each vector in $\re^{\N_d^n}$ is called a
truncated multi-sequence (tms) of degree $d$.
A tms $y \in \re^{\N_d^n}$ gives the linear functional
$\mathscr{R}_y$ acting on $\re[x]_d$ as
\be
\mathscr{R}_y\Big({\sum}_{\af \in \N_d^n}
f_\af x^\af  \Big) := {\sum}_{\af \in \N_d^n}  f_\af y_\af.
\ee
The $\mathscr{R}_y$ is called a Riesz functional.
For $f \in \re[x]_d$ and $y \in \re^{\N_d^n}$, we denote
\be \label{df:<p,y>}
\langle f, y \rangle := \mathscr{R}_y(f).
\ee
The tms $y \in \re^{\N_d^n}$ is said to admit a representing measure
supported in a set $T \subseteq \re^n$
if there exists a Borel measure $\mu$, supported in $T$,
such that $y_\af = \int  x^\af \mt{d} \mu$ for all $\af \in \N_d^n$.
This is equivalent to that
$\langle f, y \rangle = \int f(x) \mt{d} \mu$
for all $f \in \re[x]_d$.
Such $\mu$ is called a $T$-representing measure for $y$.
We refer to \cite{CurFia05,HN12,NieATKMP} for recent work
on truncated moment problems.

For a polynomial $p \in \re[x]_{2d}$, the $d$th {\it localizing matrix} of $p$
associated to a tms $y \in \re^{\N^n_{2d}}$,
is the symmetric matrix $L_p^{(d)}[y]$ such that
\be  \label{LocMat}
vec(a)^T \Big( L_p^{(d)}[y] \Big) vec(b)  = \mathscr{R}_y(p a b)
\ee
for all polynomials $a, b \in \re[x]_t$,
with $t =d - \lceil \deg(p)/2 \rceil$.
In the above, the $vec(a)$ denotes the coefficient vector of the polynomial $a$.
For instance, when $n=3$ and $p= x_1x_2-x_3^3$, for $y\in \re^{\N^3_6}$, we have
\[
L_p^{(3)}[y] = \begin{bmatrix*}[r]
y_{110}-y_{003}  &  y_{210}-y_{103}  &  y_{120}-y_{013}  &  y_{111}-y_{004} \\
y_{210}-y_{103}  &  y_{310}-y_{203}  &  y_{220}-y_{113}  &  y_{211}-y_{104} \\
y_{120}-y_{013}  &  y_{220}-y_{113}  &  y_{130}-y_{023}  &  y_{121}-y_{014} \\
y_{111}-y_{004}  &  y_{211}-y_{104}  &  y_{121}-y_{014}  &  y_{112}-y_{005} \\
\end{bmatrix*}.
\]
For the special case of constant one polynomial $p = 1$,
$L_1^{(d)}[y]$ is reduced to the so-called {\it moment matrix}
\be \label{MomMat}
M_d[y] \, := \, L_{1}^{(d)}[y].
\ee
The columns and rows of $L_p^{(d)}[y]$, as well as $M_d[y]$,
are labelled by $\af \in \N^n$ with $2|\af| + \deg(p) \leq 2d$.
For instance, for $n=3$ and $y \in \re^{\N_4^3}$, we have
\[
M_2[y]=
\begin{bmatrix*}[r]
y_{000} & y_{100}  & y_{010} & y_{001} & y_{200} & y_{110} & y_{101} & y_{020} & y_{011} & y_{002} \\
y_{100} & y_{200}  & y_{110} & y_{101} & y_{300} & y_{210} & y_{201} & y_{120} & y_{111} & y_{102} \\
y_{010} & y_{110}  & y_{020} & y_{011} & y_{210} & y_{120} & y_{111} & y_{030} & y_{021} & y_{012} \\
y_{001} & y_{101}  & y_{011} & y_{002} & y_{201} & y_{111} & y_{102} & y_{021} & y_{012} & y_{003} \\
y_{200} & y_{300}  & y_{210} & y_{201} & y_{400} & y_{310} & y_{301} & y_{220} & y_{211} & y_{202} \\
y_{110} & y_{210}  & y_{120} & y_{111} & y_{310} & y_{220} & y_{211} & y_{130} & y_{121} & y_{112} \\
y_{101} & y_{201}  & y_{111} & y_{102} & y_{301} & y_{211} & y_{202} & y_{121} & y_{112} & y_{103} \\
y_{020} & y_{120}  & y_{030} & y_{021} & y_{220} & y_{130} & y_{121} & y_{040} & y_{031} & y_{022} \\
y_{011} & y_{111}  & y_{021} & y_{012} & y_{211} & y_{121} & y_{112} & y_{031} & y_{022} & y_{013} \\
y_{002} & y_{102}  & y_{012} & y_{003} & y_{202} & y_{112} & y_{103} & y_{022} & y_{013} & y_{004} \\
\end{bmatrix*}.
\]

Suppose $g := (g_1, \ldots, g_m)$ is a tuple of polynomials in $\re[x]_{2d}$.
Consider the cone of tms of degree $2d$
\be \label{mom:S(g):2d}
\mathscr{S}(g)_{2d} :=
\Big \{
\left. y \in \re^{ \N^n_{2d} } \right|
M_d[y] \succeq 0, \,  L_{g_1}^{(d)}[y] \succeq 0, \ldots,
L_{g_m}^{(d)}[y] \succeq 0
\Big \}.
\ee
It is a closed convex cone in $\re^{ \N^n_{2d} }$.
Consider the projection map:
\be \label{map:pi}
\pi: \re^{ \N^n_{2d} } \to \re^n, \quad
y \mapsto u=(y_{e_1}, \ldots, y_{e_n}).
\ee
Let $K$ be the semialgebraic set as in \reff{set:K:g>=0}, then
\be \label{K:subset:pi}
K \subseteq \pi \Big( \mathscr{S}(g)_{2d} \cap \{ y_0 = 1\} \Big).
\ee
This is because for each $u \in K$, the tms
$y :=[u]_{2d}$ belongs to $\mathscr{S}(g)_{2d}$ and $\pi(y) = u$.
Therefore, the linear section $\{ y_0 = 1\}$
of the cone $\mathscr{S}(g)_{2d}$ is a lifted convex relaxation of the set $K$.
The cone $\mathscr{S}(g)_{2d}$ and the quadratic module
$Q(g)_{2d}$ are dual to each other. This is because
$\langle f, y \rangle \geq 0$ for all
$f \in Q_{2d}(g)$ and $y \in \mathscr{S}(g)_{2d}$.
We refer to \cite{LasBk15,Lau09,Nie-LinOpt}
for these basic properties.
Interestingly, the containment in \reff{K:subset:pi} is an equality
when the polynomials $g_i$ are SOS-concave \cite{HelNie10}.

There exists much work for polynomial optimization.
The Lasserre type Moment-SOS hierarchy of
relaxations were introduced in \cite{Las01}.
The Moment-SOS hierarchy was proved to have finite convergence
under some optimality conditions~\cite{dKlLau11,Las09,Nie-opcd}.
The flat extension or flat truncation condition
can be used to certify its convergence \cite{CurFia05,HenLas05,Nie-ft}.
For unconstrained optimization, the performance of
the standard SOS relaxations was studied in \cite{ParMP,PS03}.
For the special case of finite feasible sets,
the convergence was studied in \cite{LLR08,Lau07,NieReVar}.
We refer to \cite{BPT13,LasBk15,Lau09,Scheid09}
for more detailed introductions to polynomial optimization.

\section{The sample average optimization}
\label{sc:saopt}

Let $\xi^{(1)}, \ldots, \xi^{(N)}$ be given samples
for the random vector $\xi$. Consider the sample average function
\[
f_N(x) \, := \, \frac{1}{N} \sum_{k=1}^N F(x, \xi^{(k)}).
\]
When $F(x,\xi)$ is a polynomial in $x \in \re^n$, $f_N(x)$ is also a polynomial in $x$.
We assume that the feasible set $K$ is given as in \reff{set:K:g>=0},
for a tuple $g:=(g_1,\ldots, g_m)$ of polynomials.
Let $d$ be the degree:
\be \label{deg:d}
d = \Big \lceil \frac{1}{2} \max\{\deg(f_N),
\deg(g_1), \ldots, \deg(g_m)\} \Big \rceil.
\ee
Instead of solving \reff{SAA} directly,
we propose to solve the sample average optimization with perturbation
\be \label{pop:SAAR}
\left\{ \baray{rl}
\min & f_N(x) + \eps   \| [x]_{2d} \| \\
\st & g_1(x)\geq 0, \ldots, g_m(x) \geq 0,
\earay \right.
\ee
for a small parameter $\eps > 0$.
The Lasserre type moment relaxation can be applied to solve \reff{pop:SAAR}.
Recall the notation $[x]_{2d}$, $M_d[y]$,
$L_{g_i}^{(d)}[y]$ as in Section~\ref{sc:pre}. Observe that
\[
f_N(x) = \langle f_N, [x]_{2d} \rangle, \quad
M_d\big[ [x]_{2d} \big] \succeq 0, \quad
L_{g_i}^{(d)} \big[ [x]_{2d} \big] \succeq 0
\]
for all $x\in K$ and all $i=1,\ldots,m$.
If we replace the monomial vector $[x]_{2d}$ by a tms
$y \in \re^{ \N^n_{2d} }$, then \reff{pop:SAAR}
is relaxed to the following convex optimization
\be \label{PSAA:moment}
\left\{ \baray{rl}
\min & \langle f_N,y\rangle+\epsilon \|  y\|\\
\st & M_d[y] \succeq 0, \, L_{g_i}^{(d)} [ y] \succeq 0 \,(i=1,\ldots,m), \\
    & y_0=1, \, y \in \re^{ \N^n_{2d} }.
\earay \right.
\ee
It is a semidefinite program, with a norm function in the objective.
The relaxation \reff{PSAA:moment} is said to be {\it tight} if its optimal value
is the same as that of (\ref{pop:SAAR}).
In this paper, we choose $\|\cdot\|$ to be the standard Euclidean norm,
but any other kind of vector norms can also be used.
The equality constraint $y_0=1$ means that the first entry of $y$ is equal to one.
The set of all $y$ satisfying linear matrix inequalities in \reff{PSAA:moment}
is just the cone $\mathscr{S}(g)_{2d}$, defined as in \reff{mom:S(g):2d}.
The cone $\mathscr{S}(g)_{2d}$ and the truncated quadratic module $Q(g)_{2d}$
are dual to each other.
Therefore, the Lagrange function for \reff{PSAA:moment} is
\[
\baray{rcl}
\mathcal{L}(y, q, \gamma) & = &  \langle f_N,y\rangle+\epsilon \|  y\|
- \langle q, y \rangle - \gamma (y_0-1) \\
& = & \langle f_N -q-\gamma,y\rangle+\epsilon \|  y\| + \gamma,
\earay
\]
for dual variables $q \in Q(g)_{2d}$ and $\gamma \in \re$.
The function $\mathcal{L}(y, q, \gamma)$ has a finite minimum value
for $y \in \re^{ \N^n_{2d} }$ if and only if
\[
\| vec( f_N -q-\gamma ) \| \leq \epsilon,
\]
for which case the minimum value is $\gamma$.
(The $vec(p)$ denotes the coefficient vector of $p$.)
Therefore, the dual optimization problem of \reff{PSAA:moment} is
\be \label{PSAA:sos}
\left\{ \baray{rl}
\max & \gamma \\
\st & f_N - p - \gamma \, \in Q(g)_{2d},  \\
    & \| vec(p) \|\leq \epsilon, p \in \re[x]_{2d}.
\earay \right.
\ee
Because the sample average $f_N(x)$ is only an approximation for $f(x)$,
it is possible that there is no scalar $\gamma$ such that
$f_N - \gamma \in Q(g)_{2d}$. The perturbation term
$\epsilon \| y\|$ in \reff{PSAA:moment} motivates us to
find the maximum $\gamma$ such that
$f_N - p - \gamma \in Q(g)_{2d}$,
for some polynomial $p$ whose coefficient vector has a small norm.
This leads to the following algorithm.

\begin{alg} \label{alg:PSAA}
Generate samples $\xi^{(1)},\ldots,\xi^{(N)}$,
according to the distribution of $\xi$.
Choose a small perturbation parameter $\eps > 0$.

\bit

\item [Step~1] Compute the sample average $f_N = N^{-1}\sum_{k=1}^NF(x,\xi^{(k)})$.

\item [Step~2] Solve the semidefinite relaxation problem \reff{PSAA:moment}.
If \reff{PSAA:moment} is infeasible, increase the value of $\eps$
(e.g., let $\eps := 2 \eps$), until \reff{PSAA:moment} has a minimizer,
which we denote as $y^*$.

\item [Step~3]
Let $u = \pi (y^*)$, where $\pi$ is the projection map in \reff{map:pi},
or equivalently, let
\[
u = (y^*_{e_1}, \ldots, y^*_{e_n} ).
\]
Output $u$ as a candidate minimizer for the
sample average optimization with perturbation \reff{pop:SAAR}, and stop.

\eit

\end{alg}

For $\eps >0$, the minimizer of the relaxation~\reff{PSAA:moment}
is always unique (if it exists), because its
objective is strictly convex.
Our numerical experiments demonstrate that
Algorithm~\ref{alg:PSAA} is efficient for solving \reff{pop:SAAR}.

\begin{theorem}
Assume that $u^*$ is a minimizer of (\ref{pop:SAAR})
and $y^*$ is a minimizer of \reff{PSAA:moment}.
Then, for $\eps >0$, the relaxation (\ref{PSAA:moment})
is tight if and only if $\rank\, M_d[y^*]=1$.
In particular, for the case $\rank\, M_d[y^*]=1$,
the point $u = \pi(y^*)$ is a minimizer of \reff{pop:SAAR}.
\end{theorem}
\begin{proof}
Let $\vartheta_1, \vartheta_2$ be optimal values of (\ref{pop:SAAR})
and (\ref{PSAA:moment}) respectively.

``$\Leftarrow$''
It is clear that $\vartheta_1 \geq \vartheta_2$. If $rank M_d[y^*]=1$,
then for $u = \pi(y^*)$ one can show that
$M_d[y^*] = [u]_d ([u]_d)^T$. Hence, $y^* =[u]_{2d}$,
$\langle f_N, y^*\rangle = f_N(u)$, and each $g_i(u) \geq 0$
(see \cite{HenLas05,Nie-ft}).
So, $u$ is a feasible point of (\ref{pop:SAAR}) and
\[
\vartheta_1\leq f_N(u)+\epsilon\|[u]_{2d}\| =
\langle f_N, y^*\rangle+\epsilon \|y^*\| = \vartheta_2.
\]
Therefore, $\vartheta_1=\vartheta_2$, $u$ is a minimizer of (\ref{pop:SAAR}),
and the relaxation (\ref{PSAA:moment}) is tight.

``$\Rightarrow$'' Let $\tilde{y} :=[u^*]_{2d}$, then
$f_N(u^*)=\langle f_N,\tilde{y}\rangle$ and $\| [u^*]_{2d} \| = \| \tilde{y} \|$.
If the relaxation (\ref{PSAA:moment}) is tight,
then $\vartheta_1=\vartheta_2$ and $\tilde{y}$ is a minimizer of (\ref{PSAA:moment}).
For $\eps > 0$, the objective of (\ref{PSAA:moment}) is strictly convex,
so its minimizer must be unique.
Hence, $\tilde{y} = y^*$ and
\[
M_d[y^*] \, = \,  M_d[\tilde{y}]  \, = \, [u^*]_{d} ([u^*]_{d})^T.
\]
Therefore, $\rank \, M_d[y^*] = \rank \, M_d[\tilde{y}] = 1$.
\end{proof}

When the sample average $f_N(x)$ is unbounded from below
on the feasible set $K$,
the moment relaxation \reff{PSAA:moment} might still be
unbounded from below if $\eps >0$ is small.
However, if $\eps >0$ is big,
then \reff{PSAA:moment} must be feasible and has a minimizer.
Indeed, we have the following theorem.

\begin{theorem}
Suppose the feasible set $K$ has nonempty interior.
If $\eps>0$ is big, both \reff{PSAA:moment} and \reff{PSAA:sos}
have optimizers and their optimal values are the same.
\end{theorem}
\begin{proof}
When $K$ has nonempty interior, the quadratic module
$Q(g)_{2d}$ is a closed cone (see \cite[Theorem~3.49]{Lau09})
and the cone $\mathscr{S}(g)_{2d}$ has nonempty interior.
For instance, let $\nu$ be the Gaussian measure, then the tms
\[
\hat{y} \, := \, \frac{1}{\nu(K)} \int_K [x]_{2d} \mt{d} \nu(x)
\]
is an interior point of the cone $\mathscr{S}(g)_{2d}$. In other words,
$M_d[\hat{y}] \succ 0$ and all $L_{g_i}^{(d)}[\hat{y}] \succ 0$.
This is because $\int_K p^2d\nu>0$ and $\int_K g_ip^2d\nu>0$
for all nonzero polynomials $p$. Moreover, $\hat{y}_0 = 1$.
The convex relaxation \reff{PSAA:moment} is strictly feasible
(i.e., there is a feasible $y$ such that
each matrix in (\ref{PSAA:moment}) is positive definite).
When $\eps > 0$ is big, the SOS relaxation \reff{PSAA:sos}
is also strictly feasible. For instance, for the choice
\[
\eps > \| f_N - [x]_d^T [x]_d \|, \,
\hat{p} = f_N - [x]_d^T [x]_d, \, \hat{\gamma} = 0,
\]
we have that
\[
f_N - \hat{p} - \hat{\gamma} = [x]_d^T [x]_d
\in \mbox{int} \, \Big( \Sigma[x]_{2d} \Big)
\subseteq \mbox{int} \, \Big(  Q(g)_{2d} \Big) .
\]
In the above, $\mbox{int}$ denotes the interior of a set.
Therefore, for big $\eps > 0$, both \reff{PSAA:moment}
and \reff{PSAA:sos} have strictly feasible points.
By the strong duality theorem (see \cite{BTN01,BoyVanBk}),
they have the same optimal value
and they both achieve the optimal value,
i.e., they have optimizers.
\end{proof}

In applications, however, we often choose a small $\eps > 0$,
because we expect that \reff{pop:SAAR} is a good approximation for \reff{SAA}.
In (\ref{PSAA:moment}), the value of $\epsilon$
affects the performance of \reff{PSAA:moment}.
When $\epsilon > 0$ is too small, \reff{PSAA:moment} might be unbounded
from below and has no minimizers. If $\epsilon > 0$ is big,
\reff{PSAA:moment} might give a loose approximation for \reff{SAA}.
For efficiency, we often anticipate the smallest value of $\eps$
such that \reff{PSAA:moment} is bounded from below and has a minimizer.
When $K$ has nonempty interior, the relaxation~\reff{PSAA:moment}
is strictly feasible, i.e.,
there exists $\hat{y}$ such that all the matrices
$M_d [\hat{y}]$ and $L_{g_i}^{(d)}[\hat{y}]$ are positive definite.
Therefore, the strong duality holds between
\reff{PSAA:moment} and \reff{PSAA:sos}.
To ensure that \reff{PSAA:moment} is solvable (i.e., it has a minimizer),
the dual optimization problem \reff{PSAA:sos} needs to be feasible.
Consider the optimization problem
\be \label{min||p||}
\left\{ \baray{rl}
\eps^* :=\min & \| vec(p) \| \\
\st & f_N- p - \gamma \, \in Q(g)_{2d},  \\
    &  \gamma \in \re, \,  p \in \re[x]_{2d}.
\earay \right.
\ee
The above is a convex optimization problem with semidefinite constraints.
In computational practice, we often choose $\eps >0$
in a heuristic way, e.g., $\eps = 10^{-2}$.
If such $\eps$ is not enough, we can increase its value
until \reff{PSAA:moment} performs well.

\section{Numerical Experiments}
\label{sc:num}

This section gives numerical experiments of applying Algorithm~\ref{alg:PSAA}
to solve stochastic polynomial optimization.
The computation is implemented in MATLAB R2018a,
in a Laptop with CPU 8th Generation Intel® Core™ i5-8250U and RAM 16 GB.
The moment relaxation \reff{PSAA:moment}
is solved by the software {\tt GloptiPoly~3} \cite{GloPol3},
which calls the semidefinite program solver {\tt SeDuMi} \cite{sedumi}.
The computational results are displayed with four decimal digits.
We use ``PSAA'' to denote the perturbation sample average approximation
model \reff{pop:SAAR}. Its relaxation \reff{PSAA:moment} is said to be \textit{solvable}
if it has a minimizer $y^*$. For such a case,
let $u$ be the point given in Step~3 of Algorithm~\ref{alg:PSAA}, i.e., $u = \pi(y^*)$.
Otherwise, \reff{PSAA:moment} is said to be \textit{not solvable}
and we use ``n.a.'' to indicate the relevant values are not available.
We use $f_{min}$ and $v^*$ to denote the optimal value and the minimizer of
(\ref{eq:original}) respectively. The symbol $\bar{\xi}$
stands for the sample average of the random vector $\xi\in\mathbb{R}^r$,
while $\bar{\xi}_i$ refers to its $i$th entry:
\[
\bar{\xi} = \frac{1}{N}\sum_{k=1}^N\xi^{(k)},\quad
\bar{\xi_i}=\frac{1}{N}\sum_{k=1}^N\xi_i^{(k)},\ i = 1,\cdots, r.
\]
In our numerical examples, we use the following classical
distributions for random variables (see \cite{chung2001course};
let $\delta_a$ denote the Dirac function supported at $a$):
\begin{itemize}
\item $Ber(p)$ denotes the \textit{Bernoullian distribution}
with success probability $p$, whose density function is $(1-p)\cdot \delta_0+p\cdot \delta_1$.

\item $Geo(p)$ denotes the \textit{geometric distribution}
with success probability $p$, whose density function is
$\sum_{n=0}^{\infty}(1-p)^n p \cdot\delta_n$.

\item $\mathcal{P}(\lambda)$ denotes the \textit{Poisson distribution}
with parameter $\lambda >0$, whose density function is $\sum_{n=0}^{\infty}e^{-\lambda}\frac{\lambda^n}{n!}\cdot\delta_n$.

\item $\mathcal{U}(a,b)$, with $a< b$,
denotes the \textit{uniform distribution} on $[a,b]$.

\item $\mathcal{N}(\mu,P)$ denotes the \textit{normal distribution}
with the expectation $\mu$ and covariance matrix $P$.

\end{itemize}

\begin{exmp}\label{exmp:PSAAsdp1}
Consider the stochastic optimization problem
\be \label{eq:PSAAsdp1}
\baray{rl}
\min\limits_{x \in \re^4} &  f(x) := \mathbb{E}[G(x)+H(x,\xi)] \\
\st & x_1 -1 \geq 0,\ x_2 -1/2 \geq 0,\ x_3 - 1/3 \geq 0,\ x_4 - 1/4 \geq 0,
\earay
\ee
where
\begin{align*}
    &G(x) = (x_1^2-2x_2^2)^2+x_3(2x_3^2-3x_1x_2+x_4^2)(x_4^2-3x_1x_2)-x_4x_3^3(2x_1^3-x_3^3), \\
    &H(x,\xi) = \xi_1x_3^5+\xi_2x_1^6x_4,\ \xi_1\sim \mathcal{U}(0,2),\ \xi_2\sim \mathcal{U}(0,2).
\end{align*}
We have $\mathbb{E}(\xi_1)=\mathbb{E}(\xi_2)=1$.
The constraining polynomial tuple is
\[ g = (x_1-1, x_2-1/2, x_3-1/3, x_4-1/4). \]
The feasible set $K$ is a polyhedron and the exact objective is
\[
f = (x_1^2-2x_2^2)^2+x_3(x_3^2-3x_1x_2+x_4^2)^2+x_4(x_1^3-x_3^3)^2 \in Q(g)_{8}.
\]
For the optimization \reff{eq:original}, the optimal value and the minimizer are
\[
f_{min} \,=\, 8.4455e-07,\quad
v^* \, = \,(1.0031,0.7093,1.0031,1.0622).
\]
Since the approximation polynomial $f_N$ is uniquely determined by
$\bar{\xi_1}$ and $\bar{\xi_2}$, we can pick some typical sample averages
to explore the performance of our optimization model
(\ref{pop:SAAR}) and its relaxation \reff{PSAA:moment}.
For instance, we consider the samples of $\xi$ such that
\[
|\bar{\xi_1}-\mathbb{E}(\xi_1)| \,=\, |\bar{\xi_2}-\mathbb{E}(\xi_2)| \, = \, 10^{-2}.
\]
There are four cases of the signs:
\[
\begin{array}{rl}
 \text{\rom{1}}: \bar{\xi}_1-\mathbb{E}(\xi_1)>0, \bar{\xi}_2-\mathbb{E}(\xi_2)>0,&
 \text{\rom{2}}: \bar{\xi}_1-\mathbb{E}(\xi_1)<0, \bar{\xi}_2-\mathbb{E}(\xi_2)>0,\\
\text{\rom{3}}: \bar{\xi}_1-\mathbb{E}(\xi_1)<0, \bar{\xi}_2-\mathbb{E}(\xi_2)<0, &
\text{\rom{4}}:\bar{\xi}_1-\mathbb{E}(\xi_1)>0, \bar{\xi}_2-\mathbb{E}(\xi_2)<0.
\end{array}
\]
We solve the optimization model (\ref{pop:SAAR}) for each case.
The computational results are reported in Table~\ref{tab:PSAAsdp1}.
The PSAA model (\ref{pop:SAAR})
has clear advantages for cases \rom{3} and \rom{4},
when they are compared to the unperturbed case (i.e, $\eps = 0$).
It gives reliable optimizers, while the classical SAA
model \reff{SAA} does not return good ones.
\begin{table}[htb]
\centering
\caption{Performance of PSAA for Example~\ref{exmp:PSAAsdp1}}
\label{tab:PSAAsdp1}
\begin{tabular}{|c|c|c|c|c|c|c|}
\hline
\multicolumn{2}{|c|}{$\epsilon$} & 0 & $10^{-4}$ & $10^{-3}$ & $10^{-2}$ & $10^{-1}$\\
\hline
\multirow{4}{*}{\rom{1}}
& solvable? & yes & yes & yes & yes & yes  \\ 
& time(sec.) & 1.13 & 1.34 & 1.23 & 1.23 & 1.23\\ 
& $|\langle f_N,y^*\rangle-f_N(u)|$ & 1.29e-06 & 6.41e-07 & 2.38e-07 &  1.30e-07 & 4.75e-08\\ 
& $|\langle f_N,y^*\rangle-f_{min}|$ & 2.05e-02 & 2.05e-02 & 2.07e-02 & 3.41e-02 & 3.33e-01\\
\hline
\multirow{4}{*}{\rom{2}}
& solvable? & yes & yes & yes & yes & yes  \\ 
& time(sec.) & 1.12 & 1.29 & 1.23 & 1.21 & 1.22 \\ 
& $|\langle f_N,y^*\rangle-f_N(u)|$ & 9.15e-07 & 6.41e-07 & 2.06e-07 & 1.27e-07 &  1.53e-08 \\ 
& $|\langle f_N,y^*\rangle-f_{min}|$ & 4.84e-04 & 4.86e-04 & 6.91e-04 & 1.38e-02 & 3.13e-01\\ \hline
\multirow{4}{*}{\rom{3}}
& solvable? & no & yes & yes & yes & yes\\ 
& time(sec.) & 2.14 & 2.78 & 1.40 &1.16 & 1.15 \\ 
& $|\langle f_N,y^*\rangle-f_N(u)|$ & n.a. & 1.13e+03 & 7.50e-08 & 1.25e-07& 8.22e-09\\ 
& $|\langle f_N,y^*\rangle-f_{min}|$ & n.a. & 7.14e+05 & 3.85e-02 & 7.00e-03 & 2.94e-01\\ \hline
\multirow{4}{*}{\rom{4}}
& solvable? & no & yes & yes & yes & yes\\ 
& time(sec.) & 2.19 & 2.79 & 1.40 &1.14 & 1.16 \\ 
& $|\langle f_N,y\rangle-f_N(u)|$ & n.a. & 1.35e+03 & 1.89e-07 & 1.26e-07& 1.54e-08\\ 
& $|\langle f_N,y\rangle-f_{min}|$ & n.a. & 6.52e+05 & 4.73e-04 & 1.32e-02& 3.14e-01\\ \hline
\end{tabular}
\end{table}
\end{exmp}

\begin{exmp}\label{exmp:BddSPO}
Consider the stochastic polynomial optimization
\begin{equation}\label{eq:BddSPO}
\begin{split}
\min_{x\in \mathbb{R}^4}\quad &  f(x) := \mathbb{E}[F(x)+H(x,\xi)] \\
        \text{s.t.}\quad
        &x_1x_3+1\geq x_2^2+x_4^2,\ x_2x_3-x_1x_4+2\geq 0, \\
        & x_1^3+x_2^3+x_3^3+x_4^3\leq 8, \\
        & x_1\geq 0,\ x_2\geq 0,\ x_3\geq 0,\ x_4\geq 0,
\end{split}
\end{equation}
where
\begin{align*}
    & G(x) = x_1^2x_2^2+x_2^2x_3^2+(1-x_2x_3)^2+(3-x_1x_4)^2+x_1x_2x_3x_4, \\
    & H(x,\xi)=\xi_1x_1x_2^2x_3+\xi_2x_2^2x_4^2, \\
    & \xi\sim \mathcal{N}\left(
    \begin{bmatrix}
    -0.41\\-2.50
    \end{bmatrix},
    \begin{bmatrix}
    1 & 0\\
    0 & 1
    \end{bmatrix}
    \right).
\end{align*}
If $f(x)$ is evaluated exactly,
the optimal value and the minimizer of \reff{eq:original} are
\[
f_{min} = 1.0655,\quad
v^* = (1.5829,\, 0.6427,\, 0.9316,\, 1.4358).
\]
As in Example~\ref{exmp:PSAAsdp1},
we explore the performance of \reff{pop:SAAR}-\reff{PSAA:moment}
for the following cases of samples:
\begin{align*}
\text{\rom{1}}:\ & \bar{\xi_1} = \mathbb{E}(\xi_1)-10^{-2},\ \bar{\xi_2} = \mathbb{E}(\xi_2)-10^{-2},\\
\text{\rom{2}}:\ & \bar{\xi_1}=\mathbb{E}(\xi_1)-10^{-2},\ \bar{\xi_2} = \mathbb{E}(\xi_2),\\
\text{\rom{3}}:\ & \bar{\xi_1} = \mathbb{E}(\xi_1),\ \bar{\xi_2} = \mathbb{E}(\xi_2)-10^{-2}.
\end{align*}
The computational results are reported in Table~\ref{tab:BddSPO}.
The PSAA model (\ref{pop:SAAR})
performs better than the classical SAA model \reff{SAA} (i.e. $\epsilon=0$).
For all these cases, (\ref{pop:SAAR}) gives more reliable optimizers.
Moreover, solving the relaxation \reff{PSAA:moment}
costs less computational time for cases \rom{1} and \rom{2}.

\begin{table}[htb]
\centering
\caption{Performance of PSAA for Example \ref{exmp:BddSPO}}
\label{tab:BddSPO}
\begin{tabular}{|c|c|c|c|c|c|}
\hline
\multicolumn{2}{|c|}{$\epsilon$} & 0 & $10^{-4}$ & $10^{-3}$ & $10^{-2}$\\ \hline
\multirow{4}{*}{\rom{1}}
& solvable? & yes  & yes & yes & yes\\ 
& time(sec.) & 0.37  & 0.29 & 0.15 & 0.26\\ 
& $|\langle f_N,y^*\rangle-f_N(u)|$ &  1.81e+05  & 1.29e-08& 5.28e-09& 9.53e-09\\ 
& $|\langle f_N,y^*\rangle-f_{min}|$ & 1.81e+05  & 1.48e-02 & 1.48e-02 & 1.47e-02 \\
\hline
\multirow{4}{*}{\rom{2}}
& solvable? & yes & yes& yes& yes\\ 
& time(sec.) & 0.23 & 0.12 & 0.14 & 0.12\\ 
& $|\langle f_N,y^*\rangle-f_N(u)|$ & 1.87e+05 & 1.29e-08 & 5.61e-09 & 9.64e-09\\ 
& $|\langle f_N,y^*\rangle-f_{min}|$ & 1.87e+05 & 6.13e-03 & 6.12e-03 & 6.12e-03\\
\hline
\multirow{4}{*}{\rom{3}}
& solvable? & yes & yes & yes & yes\\ 
& time(sec.) & 0.12 & 0.13 &0.10 & 0.09\\ 
& $|\langle f_N,y^*\rangle-f_N(u)|$ & 5.58e-02 & 1.29e-08& 5.61e-09& 9.70e-09\\ 
& $|\langle f_N,y^*\rangle-f_{min}|$ & 1.40e-02  & 8.56e-03 & 8.56e-03 & 8.55e-03  \\ \hline
\end{tabular}
\end{table}
\end{exmp}

\begin{exmp}\label{exmp:MoreEx1}
Consider the stochastic optimization
\be \label{eq:MoreEx1}
\begin{split}
\min_{x\in \mathbb{R}^2}\quad & \mathbb{E}\big\{F(x,\xi)=\xi_1 x_1^4x_2^2+\xi_2 x_1^2x_2^4-\xi_3 x_1x_2^3+\xi_4 x_1x_2\big\}\\
\text{s.t.}\quad & 0\leq x_1\leq 2,\ x_1+x_2\leq 4,\ x_1x_2\leq 8,
\end{split}
\ee
where
\[\xi\sim\mathcal{N}
\left(
\begin{bmatrix}
1\\1\\3\\1
\end{bmatrix},
\begin{bmatrix}
1 & 0.1 & 0.3 & 0.2\\
0.1 & 1 & 0.4 & 0.3\\
0.3 & 0.4 & 1 & 0.2\\
0.2 & 0.3 & 0.2 & 1
\end{bmatrix}\right).\]
The exact objective $f(x) = \mathbb{E}[F(x,\xi)]$
can be evaluated as
\[
f(x) = x_1^4 x_2^2+x_1^2 x_2^4-3x_1 x_2^3+x_1 x_2.
\]
For the optimization (\ref{eq:original}), its optimal value and minimizer are
\[
f_{min} = -27.8444,\quad v^* = (0.3442, \,3.6558).
\]
For convenience, we use ``$\bar{\xi}=\mathbb{E}(\xi)+o(10^{-3})$''
to denote a random sample average of size $N=1000$
with error in the order of $10^{-3}$. Consider two cases:
\[
\text{\rom{1}}: \bar{\xi}=\mathbb{E}(\xi),\quad
\text{\rom{2}}:\ \bar{\xi}=\mathbb{E}(\xi)+o(10^{-3}).
\]
The numerical results are reported in Table~\ref{tab:MoreEx1}.
The PSAA model (\ref{pop:SAAR})
gives a better optimizer than the unperturbed one, for both cases.
\begin{table}[htb]
\centering
\caption{Performance of PSAA for Example \ref{exmp:MoreEx1}}
\label{tab:MoreEx1}
\begin{tabular}{|c|c|c|c|c|}
\hline
\multicolumn{2}{|c|}{$\epsilon$}& $0$ & $10^{-4}$ & $10^{-3}$\\
\hline
\multirow{4}{*}{\rom{1}}
& solvable? & yes & yes & yes \\ 
& time(sec.) & 0.17 & 0.13 & 0.12\\ 
& $|\langle f_N,y^*\rangle-f_N(u)|$ & 1.86e+03 & 4.18e-06 & 3.55e-06\\ 
& $|\langle f_N,y^*\rangle-f_{min}|$ & 3.44e+01 & 2.09e-04 & 2.30e-02\\
\hline
\multirow{4}{*}{\rom{2}}
& solvable? & yes & yes & yes\\ 
& time(sec.) & 0.13 & 0.14 & 0.16\\ 
& $|\langle f_N,y^*\rangle-f_N(u)|$ & 1.85e+03 & 4.17e-06 & 3.56e-06\\ 
& $|\langle f_N,y^*\rangle-f_{min}|$ & 3.44e+01 & 1.29e-02 & 9.87e-03 \\
\hline
\end{tabular}
\end{table}
\end{exmp}

\begin{exmp}\label{exmp:MoreEx3}
Consider the stochastic polynomial optimization
\begin{equation}
\begin{split}
\min_{x\in \mathbb{R}^3}\quad & f(x) :=\mathbb{E}[G(x)+H(x,\xi) ]\\
\text{s.t.}\quad & x_1\geq 0,\ x_2\geq 0,\ x_3\geq 0,\ 1-\mathbf{1}^Tx\geq 0,
\end{split}
\end{equation}
where
\begin{align*}
& G(x) = x_1^4+x_1x_2x_3+x_3(1-x_1^2-x_2^2), \\
& H(x,\xi) = 2\xi_1 x_2^4-4\xi_1 x_1^2x_2^2-\xi_2x_1x_2,\\
& \xi_1 \sim Ber(0.5),\quad \xi_2\sim Geo(0.5).
\end{align*}
The feasible set $K$ is a simplex, which is compact and satisfies the archimedean condition.
For all samples $\xi^{(i)}$, the sample average $f_N(x)$
is bounded from below on $K$ and it has a minimizer. For this example,
$\mathbb{E}(\xi_1)=0.5$ and $\mathbb{E}(\xi_2)=2$, so
\[
f(x) = (x_1^2-x_2^2)^2+x_3(1-x_1^2-x_2^2)-x_1x_2(2-x_3).
\]
The optimal value and minimizer of \reff{eq:original} are
\[
f_{min} = -0.5 , \quad  v^* = (0.5, \, 0.5, \,0).
\]
We consider two cases of samples
\begin{align*}
& \text{\rom{1}}:\ \bar{\xi_1}=\mathbb{E}(\xi_1)+10^{-3},\
\bar{\xi_2}=\mathbb{E}(\xi_2),\ \epsilon^* \approx 0.001155; \\
& \text{\rom{2}}:\ \bar{\xi_1}=\mathbb{E}(\xi_1),\
\bar{\xi_2}=\mathbb{E}(\xi_2)+10^{-3},\ \epsilon^* \approx 7.5875\times 10^{-9}.
\end{align*}
In the above, $\epsilon^*$ is the minimum value of (\ref{min||p||}).
\begin{table}[htb]
\centering
\caption{Performance of PSAA for Example~\ref{exmp:MoreEx3}.}
\label{tab:MoreEx3}
\begin{tabular}{|c|c|c|c|c|c|}
\hline
\multirow{5}{*}{\rom{1}}
& $\epsilon$ & $0$ & $0.0012$ & $0.004$ & $0.008$\\ \cline{2-6}
& solvable? & no & yes & yes& yes\\ 
& time(sec.) & 0.12 & 0.10 & 0.07 & 0.09\\ 
& $|\langle f_N,y^*\rangle-f_N(u)|$ & n.a. & 5.31e-03 & 3.36e-04 & 1.09e-04\\ 
& $|\langle f_N,y^*\rangle-f_{min}|$ & n.a. & 5.44e-03 & 4.61e-04 & 2.34e-04\\
\hline
\multirow{5}{*}{\rom{2}}
& $\epsilon$ & $0$ & $10^{-4}$ & $10^{-3}$ & $10^{-2}$\\ \cline{2-6}
& solvable? & yes & yes & yes& yes\\ 
& time(sec.) & 0.08 & 0.08 & 0.07 & 0.06\\ 
& $|\langle f_N,y^*\rangle-f_N(u)|$ & 5.03e-09 & 1.84e-09 & 1.61e-09 & 1.86e-09\\ 
& $|\langle f_N,y^*\rangle-f_{min}|$ & 2.50e-04 & 2.50e-04 & 2.50e-04 & 2.50e-04
\\ \hline
\end{tabular}
\end{table}
The numerical results are reported in Table~\ref{tab:MoreEx3}.
The PSAA model (\ref{pop:SAAR}) performs very well for both cases.
Compared with the classical SAA model \reff{SAA} (i.e., $\eps = 0$),
it has quite clear advantages for case \rom{1}.
It successfully returned a good minimizer,
while \reff{SAA} is unbounded from below and does not return a minimizer.
\end{exmp}

\begin{exmp}\label{exmp:UcSPO}
Consider the unconstrained stochastic optimization
\be \label{eq:UcSPO}
\min_{x\in\mathbb{R}^4} \,  \mathbb{E} [G(x)+H(x,\xi)]
\ee
where
\begin{align*}
&\xi\sim\mathcal{P}(2),\ G(x)=(x_3-x_4)^4+(x_1+x_2)^4+x_1^2+x_2^2+x_3^2+x_4^2, \\
&H(x,\xi)=\xi-(\xi^2-2\xi)(x_1-x_4)-2(
\xi-1)(x_3-x_4)^2(x_1+x_2)^2.
\end{align*}
Evaluating the expectation, we get $\mathbb{E}(\xi)=2, \mathbb{E}(\xi^2)=6$ and
\[
f(x) = [(x_3-x_4)^2-(x_1+x_2)^2]^2+(x_1-1)^2+(1+x_4)^2+x_2^2+x_3^2.
\]
The optimal value and minimizer of \reff{eq:original} are
\[
f_{min} = 0,\quad v^*=(1,\,0,\,0,\,-1).
\]
%
%
Approximate $\xi,\xi^2$ by their sample averages $\bar{\xi}, s_2$, i.e., $s_2=\frac{1}{N}\sum_{k=1}^N(\xi^{(k)})^2$. We make samples of different sizes and compute $\epsilon^*$ in (\ref{min||p||}) for each case.
\begin{table}[htb]
\centering 
\caption{The values of $\eps^*$ for Example~\ref{exmp:UcSPO}.}
\begin{tabular}{|c|c|c|c|c|}
\hline
& \rom{1} & \rom{2} & \rom{3} & \rom{4} \\ \hline
$N$ & $500$ & $1000$ & $5000$ & $10000$\\ 
$\bar{\xi}$ & 2.11 & 1.96 & 2.01 & 2.02\\ 
$s_2$ & 6.43 & 5.71 & 6.13 & 6.07\\
\hline
$\epsilon^*$ & 0.807543 & 3.3618e-10 & 0.073413 & 0.146826\\
\hline
\end{tabular}
\end{table}
We focus on cases \rom{3} and \rom{4}.
\begin{table}[htb]
\centering
\caption{Performance of PSAA for Example~\ref{exmp:UcSPO}.}
\label{tab:UcSPO1}
\begin{tabular}{|c|c|c|c|c|c|c|}
\hline
 & \multicolumn{3}{c|}{\rom{3}} & \multicolumn{3}{c|}{\rom{4}}\\
\hline
$\epsilon$ & 0 & $\epsilon^*$ & 0.1 & 0 & $\epsilon^*$ & 0.2\\
\hline
solvable? & no & yes & yes & no & yes & yes \\ 
time(sec.) & 0.23 & 0.20 & 0.11 & 0.10 & 0.16 & 0.07\\
\hline
$|\langle f_N,y^*\rangle-f_N(u)|$ & n.a. & 5.30e+01 & 2.02e-01 & n.a. & 6.47e+01 & 2.28e-01\\ 
$|\langle f_N,y\rangle-f_{min}|$ & n.a. & 5.39e+01 & 3.90e-01 & n.a. & 6.49e+01 & 2.96e-01\\ 
$\|u-v^*\|$ & n.a. & 3.09e-02 & 1.28e-01 & n.a. & 5.86e-02 & 2.73e-01\\ \hline
$u$ & n.a. & \small$\begin{bmatrix}1.0216\\
    0.0035\\
    0.0035\\
   -1.0216\end{bmatrix}$ &
   \small$\begin{bmatrix}
   0.9102\\
    0.0071\\
    0.0071\\
   -0.9102\\
   \end{bmatrix}$
    & n.a. & \small$\begin{bmatrix}0.9591\\
    0.0059\\
    0.0059\\
   -0.9591\end{bmatrix}$ &
   \small$\begin{bmatrix}
   0.8070\\
    0.0085\\
    0.0085\\
   -0.8070
   \end{bmatrix}$\\
   \hline
\end{tabular}
\end{table}
The computational results are reported in Table~\ref{tab:UcSPO1}.
The perturbation term in the PSAA model \reff{pop:SAAR}
makes a big difference for computing reliable minimizers.
The PSAA model returned minimizers
that are close to the optimizer of (\ref{eq:original}),
while the classical SAA model (i.e., $\epsilon=0$) is unbounded from below
and fails to return a minimizer.
\end{exmp}

\begin{exmp}\label{exmp:nonlinear1}
Consider the stochastic polynomial optimization
\begin{equation}\label{eq:nonlinear1}
\begin{split}
\min_{x\in \mathbb{R}^2}\quad & f(x):=\mathbb{E}[G(x)+H(x,\xi)]\\
\text{s.t.}\quad & x_1-1\geq 0,\ x_2\geq 0,\ 2-\mathbf{1}^Tx\geq 0
\end{split}
\end{equation}
where
\begin{align*}
& G(x) = x_1^4+x_2^4+x_1x_2-2(x_1+x_2)+1,\\
& H(x,\xi) = \xi_3x_1x_2\big[\xi_1x_1+\xi_2 x_2-(\xi_1+\xi_2)x_1x_2\big],\\
& \xi\sim\mathcal{N}\left(
\begin{bmatrix}0\\0\\1\end{bmatrix},
\begin{bmatrix}1 & 0 & 1\\
0 & 1 & 1\\
1 & 1 & 3\end{bmatrix}
\right).
\end{align*}
Note that $\mathbb{E}(\xi_1\xi_3)=\mathbb{E}(\xi_2\xi_3)=1$ and the exact objective is
\[f = (x_1^2-x_2^2)^2+(x_1x_2-1)(x_1+x_2)+(x_1-1)(x_2-1). \]
Its optimal value and minimizer are
\[
f_{min} = -0.2500,\quad v^* = (1.0000,\,0.7071 ).
\]
Denote $s_1:=\sum_{k=1}^N \xi_1^{(k)}\xi_3^{(k)},\
s_2:=\sum_{k=1}^N \xi_2^{(k)}\xi_3^{(k)}$. We make samples of different sizes and compute $\epsilon^*$ in (\ref{min||p||}) for each case.
The values of $\epsilon^*$ are given in Table~\ref{tab:eps:4.6}.
\begin{table}[htb]
\centering
\caption{The values of $\eps^*$ for Example~\ref{exmp:nonlinear1}.}
\label{tab:eps:4.6}
\begin{tabular}{|c|c|c|c|c|c|}
\hline
& \rom{1} & \rom{2} & \rom{3} & \rom{4}\\ \hline
$N$ & $500$ & $1000$ & $5000$ & $10000$\\ 
$s_1$ & 0.98 & 1.08 & 1.01 & 0.97\\
$s_2$ & 1.06 & 0.96 & 1.02 & 0.96\\
\hline
$\epsilon^*$ & 0.023094 & 0.023094 & 0.017321 & 1.1076e-09
\\
\hline
\end{tabular}
\end{table}
We focus on cases \rom{2} and \rom{3}, for which
the numerical results are reported in Table~\ref{tab:nonlinear1}.
\begin{table}[htb]
\centering
\caption{Performance of PSAA for Example~\ref{exmp:nonlinear1}.}
\label{tab:nonlinear1}
\begin{tabular}{|c|c|c|c|c|c|c|}
\hline
 & \multicolumn{3}{c|}{\rom{2}} & \multicolumn{3}{c|}{\rom{3}}\\
\hline
$\epsilon$ & 0 & $\epsilon^*$ & 0.05& 0 & $\epsilon^*$ & 0.05\\
\hline
solvable? & no & yes & yes & no & yes & yes \\ 
time(sec.) & 0.06 & 0.21 & 0.06 & 0.06 & 0.13 & 0.05\\
\hline
$|\langle f_N,y^*\rangle-f_N(u)|$ & n.a. & 1.13e+02 & 1.36e-02 & n.a. & 5.65e+00 &6.25e-03\\ 
$|\langle f_N,y\rangle-f_{min}|$ & n.a. & 1.13e+02 & 4.21e-03 & n.a. & 5.65e+00 & 2.80e-03\\ 
$\|u-v^*\|$ & n.a. & 1.20e-03 & 1.85e-02 & n.a. & 6.08e-03 & 2.58e-02\\
\hline
$u$ & n.a. & \small$\begin{bmatrix}1.0000\\
    0.7059\end{bmatrix}$ & \small$
    \begin{bmatrix}
    1.0000\\
    0.6886
    \end{bmatrix}$ & n.a. & \small$\begin{bmatrix}1.0000\\
    0.7010\end{bmatrix}$ & \small$
    \begin{bmatrix}
    1.0000\\
    0.6813
    \end{bmatrix}$\\
   \hline
\end{tabular}
\end{table}
While the optimization (\ref{SAA}) is unbounded from below via the classical SAA model, our PSAA model (\ref{pop:SAAR}) provides good lower bounds and returns reliable minimizers for all cases.
\end{exmp}

\begin{exmp}\label{exmp:nonlinear2}
Consider the stochastic polynomial optimization
\begin{equation}\label{eq:nonlinear2}
\begin{split}
\min_{x\in\mathbb{R}^3}\quad & f(x)=\mathbb{E}[F(x,\xi)] \\
\text{s.t.}\quad & x_1-1\geq 0,\ x_2-1\geq 0,\ x_3-1\geq 0,\\
& 8-x_1x_2x_3\geq 0
\end{split}
\end{equation}
where $\xi\sim \mathcal{U}(0,2)$ and
\[F(x,\xi) = (-\xi^3+3\xi_2^2-\xi)x_1x_2x_3(x_1+x_2+x_3)+(\xi^3-\xi)(x_1x_2+x_2x_3+x_1x_3).
\]
One can see that $\mathbb{E}(\xi)=1, \mathbb{E}(\xi^2)=4/3$, $\mathbb{E}(\xi^3) = 2$, and the exact objective
\[f=x_1x_2x_3(x_1+x_2+x_3)+x_1x_2+x_2x_3+x_1x_3.\]
Its optimal value and minimizer are
\[f_{min}=6,\quad v^* = (1,\,1,\,1).\]
As in Examples~\ref{exmp:UcSPO}-\ref{exmp:nonlinear1},
we approximate $\xi,\xi^2,\xi^3$ by sample averages where
$s_2 = \sum_{k=1}^N\big(\xi^{(k)}\big)^2,\ s_3 = \sum_{k=1}^N\big(\xi^{(k)}\big)^3$.
Several samples of different sizes are made.
The values of $\epsilon^*$ in (\ref{min||p||}) are given in Table~\ref{tab:SamNonlinear2}.
\begin{table}[htb]
\centering
\caption{The values of $\eps^*$ for Example~\ref{exmp:nonlinear2}.}
\label{tab:SamNonlinear2}
\begin{tabular}{|c|c|c|c|c|}
\hline
& \rom{1} & \rom{2} & \rom{3} \\ \hline
$N$ & $500$ & $1000$ & $5000$\\ 
$\bar{\xi}$ & 0.99 & 1.03 & 1.00\\
$s_2$ & 1.32 & 1.38 & 1.33\\
$s_3$ & 1.97 & 2.09 & 1.99\\
\hline
$\epsilon^*$ & 0.508637 & 0.518810 & 0.508637
\\
\hline
\end{tabular}
\end{table}
We apply the classical SAA model (i.e., $\epsilon=0$) and the PSAA model
(i.e., $\epsilon=\epsilon^*$) to cases \rom{1}, \rom{2} and \rom{3}.
\begin{table}[htb]
\centering
\caption{Performance of PSAA for Example~\ref{exmp:nonlinear2}.}
\label{tab:nonlinear2}
\begin{tabular}{|c|c|c|c|c|c|c|}
\hline
 & \multicolumn{2}{c|}{\rom{1}} &\multicolumn{2}{c|}{\rom{2}} & \multicolumn{2}{c|}{\rom{3}}\\
\hline
$\epsilon$ & 0 & $\epsilon^*$ & 0 & $\epsilon^*$ & 0 & $\epsilon^*$\\
\hline
solvable? & yes & yes & yes & yes & yes & yes \\ 
time(sec.) &0.07& 0.06& 0.07 & 0.06 & 0.08 & 0.06 \\
\hline
$|\langle f_N,y^*\rangle-f_N(u)|$ &7.08e+01& 1.07e-06 & 7.08e+01 & 8.44e-07 & 7.08e+01 & 1.07e-06 \\ 
$|\langle f_N,y\rangle-f_{min}|$ & 2.07e+01 & 2.00e-02& 2.07e+01 & 1.00e-02 & 2.07e+01 & 1.00e-02 \\ 
$\|u-v^*\|$ & 1.67e+00 & 2.97e-08
&1.67e+00 & 2.97e-08 & 1.67e+00 & 2.97e-08\\ \hline
$u$ & \small$\begin{bmatrix}
1.9637\\
    1.9637\\
    1.9630
\end{bmatrix}$ &\small$\begin{bmatrix}1.0000\\
    1.0000\\
    1.0000\end{bmatrix}$ &\small$\begin{bmatrix}
1.9631\\
    1.9631\\
    1.9627
\end{bmatrix}$ & \small$\begin{bmatrix}1.0000\\
    1.0000\\
    1.0000\end{bmatrix}$ & \small$\begin{bmatrix}
    1.9631\\
    1.9631\\
    1.9627
    \end{bmatrix}$ & \small$\begin{bmatrix}1.0000\\
    1.0000\\
    1.0000\end{bmatrix}$ \\
   \hline
\end{tabular}
\end{table}
The computational results are reported in the Table~\ref{tab:nonlinear2}.
The PSAA model (\ref{pop:SAAR}) performs much better than
the classical SAA model (\ref{SAA}), as (\ref{pop:SAAR})
gives more reliable minimizers for all cases.
\end{exmp}

\section{Conclusion}

This paper proposes a sample average optimization model
with a perturbation term for solving stochastic polynomial optimization.
The perturbation optimization model performs better than
the classical one without perturbations.
The Lasserre type moment relaxations are used
to solve the perturbation optimization. In particular,
we show that the moment relaxation is tight if and only if
the moment matrix of the minimizer is rank one.
Numerical experiments demonstrated advantages of
our perturbation optimization model.

\end{document}